\documentclass[12pt,a4paper]{article}
\usepackage[margin=1in]{geometry}
\usepackage{amsmath,amsthm,amssymb}
\usepackage{mathtools}
\usepackage[numbers]{natbib}
\usepackage{hyperref}

\newtheorem{definition}{Definition}
\newtheorem{proposition}{Proposition}
\newtheorem*{proposition*}{Proposition}
\newtheorem{corollary}{Corollary}
\newtheorem{theorem}{Theorem}
\newtheorem{lemma}{Lemma}
\newtheorem{remark}{Remark}
\newtheorem{observation}{Observation}
\newtheorem{question}{Question}

\DeclareMathOperator{\const}{const}

\begin{document}
\title{Growth of bilinear maps II: Bounds and orders}
\author{Vuong Bui\thanks{LIRMM, Universit\'e de Montpellier, CNRS and UET, Vietnam National University, Hanoi (\texttt{bui.vuong@yandex.ru})\\
Part of the work was supported by the Deutsche 
Forschungsgemeinschaft (DFG) Graduiertenkolleg ``Facets of Complexity'' 
(GRK 2434) at
Institut f\"ur Informatik, 
Freie Universit\"{a}t Berlin.}}
\date{}
\maketitle

\begin{abstract}
	A good range of problems on trees can be described by the following general setting: Given a bilinear map $*:\mathbb R^d\times\mathbb R^d\to\mathbb R^d$ and a vector $s\in\mathbb R^d$, we need to estimate the largest possible absolute value $g(n)$ of an entry over all vectors obtained from applying $n-1$ applications of $*$ to $n$ instances of $s$. When the coefficients of $*$ are nonnegative and the entries of $s$ are positive, the value $g(n)$ is known to follow a growth rate $\lambda=\lim_{n\to\infty} \sqrt[n]{g(n)}$. In this article, we prove that for such $*$ and $s$ there exist nonnegative numbers $r,r'$ and positive numbers $a,a'$ so that for every $n$,
	\[
		  a n^{-r}\lambda^n\le g(n)\le a' n^{r'}\lambda^n.
	\]

	While proving the upper bound, we actually also provide another approach in proving the limit $\lambda$ itself. The lower bound is proved by showing a certain form of submultiplicativity for $g(n)$. Corollaries include a lower bound and an upper bound for $\lambda$, which are followed by a good estimation of $\lambda$ when we have the value of $g(n)$ for an $n$ large enough. 
\end{abstract}

\section{Introduction}
Consider a rooted binary tree $T$, a pruned tree of $T$ is a tree obtained from $T$ by removing zero or more subtrees. Let $f(n)$ be the maximum number of pruned trees of a tree $T$ with $n$ leaves. The function $f(n)$ can be defined recursively by $f(1)=1$ and for $n\ge 2$,
\begin{equation}\label{eq:simple-g(n)}
	f(n)=1+\max_{1\le m\le n-1} f(m)f(n-m).
\end{equation}

This function was investigated for a different purpose in \cite{de2012maximum} where the growth rate $\lambda=\lim_{n\to\infty} \sqrt[n]{f(n)}$ was shown to be $1.502836801\dots$. It uses a pure combinatorial argument to show that the growth rate is the same as the rate $\lim_{m\to\infty} (a_m)^{1/2^m}$ of the double exponential sequence $a_m$ where $a_0=1$ and $a_{m}=1+a_{m-1}^2$ for $m\ge 1$. Actually, the closed-form expression of $\lambda$ is
\begin{equation}\label{eq:complicated-lambda}
	\lambda = \exp\left(\sum_{i\ge 1} \frac{1}{2^i} \log\left(1+\frac{1}{a_i^2}\right)\right).
\end{equation}

If we let $s=(1,1)$ and $*:\mathbb R^2\times\mathbb R^2\to \mathbb R^2$ be presented by $x*y=(x_1y_1+x_2y_2, x_2y_2)$ (with $x_i$ denoting the $i$-th entry of a vector $x$), and $g(n)$ be the largest possible entry of a combination of $n$ instances of $s$ using $n-1$ applications of $*$, then $f(n)$ and $g(n)$ are identical. Although this language of bilinear maps looks more complicated, it has more applications. The problem of studying $g(n)$ where $s$ is any nonnegative vector in $\mathbb R^d$ and $*$ is any bilinear operator in $\mathbb R^d\times\mathbb R^d\to\mathbb R^d$ with nonnegative coefficients was first introduced by Rote in \cite{rote2019maximum}. The original motivation is the maximum number of minimal dominating sets in a tree of $n$ leaves. This number is actually of the same order as $g(n)$ for vector $s=(0, 1, 0, 0, 0, 1)$ and the operator $*$ so that
\begin{equation}\label{eq:dominating-sets}
	\begin{pmatrix}
		x_1\\x_2\\x_3\\x_4\\x_5\\x_6
	\end{pmatrix}
	*
	\begin{pmatrix}
		y_1\\y_2\\y_3\\y_4\\y_5\\y_6
	\end{pmatrix}
	=
	\begin{pmatrix}
		x_1y_1 + x_1y_4 + x_1y_6 + x_2y_6 + x_3y_6\\
		x_2y_4\\
		x_2y_1 + x_3y_1 + x_3y_4\\
		x_4y_1 + x_4y_2 + x_4y_4 + x_4y_5 + x_6y_1 + x_6y_2\\
		x_5y_4 + x_5y_5 + x_6y_3\\
		x_6y_4 + x_6y_5
	\end{pmatrix}.
\end{equation}
The relation between minimal dominating sets and the setting of $*,s$, which is not obvious, is explained in detail in the original source \cite{rote2019maximum} by dynamic programming.

As the argument for the value of the growth rate $\lambda$ in the first example, which spans some $4$ pages, is already not so trivial, we may expect a different approach than elementary techniques for the second example, which looks fairly more complicated. By considering a sequence of combinations that follows a pattern with a computable growth rate $\lambda_0$, we obtain the lower bound $\lambda\ge \lambda_0$. In particular, the pattern in \cite{rote2019maximum} for the second example represents optimal trees with a beautiful structure composed of snowflakes. However, the harder part is in verifying the upper bound $\lambda\le\lambda_0$ for a given $\lambda_0$. One can observe that if we divide $s$ by some $\lambda_0\ge \lambda$ then the growth rate becomes at most $1$. In principle, this does not imply that the new $g(n)$ is bounded (e.g. the growth rate of a polynomial is still $1$). However, the interesting point here is that for this particular case the boundedness of the new $g(n)$ is ensured, as stated in \cite[Proposition $5.1$]{rote2019maximum} as follows.
\begin{proposition*}[Rote 2019]
	The growth rate $\lambda$ for $s$ and $*$ in \eqref{eq:dominating-sets} is the \emph{smallest} $\lambda_0$ so that all the vectors obtained by combining any number of instances of $s/\lambda_0$ using applications of $*$ are bounded.
\end{proposition*}

One of the merits of the approach in \cite{rote2019maximum} is the use of computers to assist the proof with a polytope $P$ so that $s/\lambda_0\in P$ and $P * P\subseteq P$. One can construct $P$ by starting with a set of a single point $s/\lambda_0$ and subsequently applying $*$ to all pairs of the points $x,y$ (not necessarily distinct) in the set and add the results $x*y$ to the set. In fact, we only need to maintain the vertices of the convex hull. We are done when we have certain form of convergence, e.g. something satisfying $P*P\subseteq P$. Details can be found in the original article, with certain delicate care of numerical operations as well. Although the expression of the second example looks complicated, the growth rate is simple as
\[
	\lambda=\sqrt[13]{95}.
\]

The approach was applied by Rosenfeld to address other problems of the number of different types of dominating sets, perfect codes, different types of matchings, and maximal irredundant sets in a tree. We suggest the readers to check \cite[Section 5]{rosenfeld2021growth} for this rich set of applications. One can find an application in graphs other than trees in \cite{de2012maximum} where the maximum number of cycles in an outerplanar graph is studied using function $f(n)$ in \eqref{eq:simple-g(n)}. We believe the flexibility of the setting allows applications in a more remote field.

One of the drawbacks of the polytope method is that it works only for the functions that satisfy $g(n)\le\const \lambda^n$, which does not hold in general. Moreover, the precision in calculating may be tricky, especially if $\lambda$ has a complicated nature as in \eqref{eq:complicated-lambda}. (Note that the precise values of the growth rates in \cite{rosenfeld2021growth} are all algebraic.)

In this article, we narrow the scope of the problem to positive vectors $s$ and operators $*$ with nonnegative coefficients only, since it was shown in \cite{bui2021growth} that the limit is always guaranteed:
\[
	\lambda=\lim_{n\to\infty} \sqrt[n]{g(n)}.
\]
When the requirements do not meet, there is chance that the limit does not hold (also see \cite{bui2021growth}).

The main result of the article is the following asymptotic description of $g(n)$, which is better than the mere limit, and differs from $\lambda^n$ by some polynomials. 
\begin{theorem}\label{thm:main}
	There are some nonnegative numbers $r,r'$ and positive numbers $a,a'$ so that for every $n$,
	\[
		a n^{-r}\lambda^n\le g(n)\le a' n^{r'}\lambda^n.
	\]
\end{theorem}

Throughout the work, the notation $\const$ means that we can put some positive constant in that place, which is independent of $n$. Two instances of $\const$ may present two different constants.

Note that we can assume $r,r'$ to be nonnegative since negative values can always be replaced by nonnegative ones. Actually, we can just assign to $r,r'$ the same number, say their maximum value. As the degrees are denoted differently, the readers may guess that the theorem is proved in two parts. The upper bound is simpler and proved first in Section \ref{sec:upper-bound} while the lower bound is proved in Section \ref{sec:lower-bound}. During the course of proving the upper bound, we also give a stand-alone proof for the limit $\lambda=\lim_{n\to\infty} \sqrt[n]{g(n)}$. This proof has a quite different approach than the proof in \cite{bui2021growth} though both share similar ideas. At any rate, the former gives more insights than the latter at roughly the same complexity of arguments. 

It follows from Theorem \ref{thm:main} a theoretical bound of $\lambda$ as follows. 
\begin{corollary}\label{cor:estimating-lambda}
	For any $n$, we have
	\[
		\sqrt[n]{\frac{1}{a'} n^{-r'} g(n)} \le \lambda \le \sqrt[n]{\frac{1}{a} n^r g(n)}.
	\]
\end{corollary}

As the readers can check the proof of Theorem \ref{thm:main}, the constants $a,a',r,r'$ are always computable from $*$ and $s$. Since the ratio $\sqrt[n]{(a'/a) n^{r+r'}}$ between the upper bound and the lower bound converges to $1$, we obtain a good bound when $n$ is large enough. However, the readers can see that this is merely of theoretical interest, since the constants $a,a',r,r'$ as constructed in the proof of Theorem \ref{thm:main} are so large while the exponential computation of $g(n)$ for large $n$ is not so feasible in general. On the other hand, the estimate can be seen to be not very bad, since the problem of approximating $\limsup_{n\to\infty} \sqrt[n]{g(n)}$ when $s$ is only nonnegative is known to be $NP$-hard \cite{rosenfeld2022undecidable} (the limit does not always exist in this case, so the limit superior is considered instead).

It makes sense to give an example where $g(n)$ is of order $n^r\lambda^n$ for some $r>0$. Let $s=(1,1)$ and $x*y=(x_1y_2 + x_2y_1, x_2y_2)$, then every combination of $n$ instances of $s$ gives the same result $(n,1)$, that is $g(n)=n$. A higher order can be obtained if we set $s=(1,1,1)$ and $x*y=(x_1y_2+x_2y_1, x_2y_2, x_1y_1)$, the largest third entry should be of order $n^2$ then. Other higher orders can be constructed correspondingly. The next order would be $n^4$ for $s=(1,1,1,1)$ and $x*y=(x_1y_2+x_2y_1, x_2y_2, x_1y_1,x_3y_3)$. Inspired by the construction, we would conjecture that 
\[
	g(n)\le\const n^{2^R}\lambda^n,
\]
where $R$ is the length of the longest chain in the partial order set of the components. (The readers can check the definition of this partial order set in Section \ref{sec:preliminary}, but for now we can say $R$ is less than the dimension, for a rough estimate.) The corresponding bound for nonnegative matrices $A$ is $\|A^n\|\le\const n^{R'}\lambda^n$ where $R'$ is the length of the similarly-defined partial order set of the components (see \cite{bui2022joint}). It is interesting when $R',2^R$ are linear and exponential, in corresponding to linear and bilinear maps.

Since for a matrix $A$ there exists a number $r$ so that $\const n^r {\rho(A)}^n\le \|A^n\|\le \const n^r {\rho(A)}^n$, where $\rho(A)$ denotes the spectral radius of $A$, it is natural to ask the following question.
\begin{question}\label{ques:same-power}
	Does there always exist a number $r$ so that
	\[
		\const n^r\lambda^n \le g(n) \le \const n^r\lambda^n?
	\]
\end{question}

The example of pruned trees in \eqref{eq:simple-g(n)} was confirmed to satisfy $\lambda^{n-\frac{1}{4}} < g(n) < \lambda^n$ for $n\ge 10$ in \cite[Theorem $2$]{de2012maximum}.
The example of minimal dominating sets in \eqref{eq:dominating-sets} also satisfies $\const\lambda^n \le g(n) \le \const\lambda^n$ in \cite[Theorem $1.1$]{rote2019maximum}. 

By Corollary \ref{cor:estimating-lambda}, given a value $\lambda_0$ other than $\lambda$ one can always decide whether $\lambda>\lambda_0$ or $\lambda<\lambda_0$ (regardless of complexity) since when $n$ is large enough, the value $\lambda_0$ will be to the left or to the right of the small interval containing $\lambda$. However, when $\lambda_0$ is given without being guaranteed to be different from $\lambda$, we have the following question.
\begin{question}
	Given any $\lambda_0$, is the problem checking if $\lambda\ne\lambda_0$ decidable?
\end{question}

When $s$ is allowed to be nonnegative instead of being strictly positive and $\lambda$ is defined to be $\limsup_{n\to\infty} \sqrt[n]{g(n)}$, Rosenfeld has given a negative answer to the question in \cite{rosenfeld2022undecidable}. A simpler argument is given in \cite{bui2022growth}. The original question for $s>0$ is still left open. In fact, it is interesting to treat the problem when there is no condition on the signs of the entries and the coefficients, and $g(n)$ is the largest possible norm of any vector obtained from combining $n$ instances of $s$ for some appropriate norm. Note that the growth rate is independent of the chosen norm as two norms are in a constant factor of each other.

Before proving Theorem \ref{thm:main}, we give some formal definitions in Section \ref{sec:preliminary}. There are also some results there, mostly brought from \cite{bui2021growth}. They are used in proving the lower bound in Section \ref{sec:lower-bound}. Note that the proof of the limit in Section \ref{sec:upper-bound} is stand-alone, and it does not use the results in Section \ref{sec:preliminary} but merely the definitions.

\section{Some definitions and preliminary results}
\label{sec:preliminary}
This section gives the formal definitions. Most of them can be also found in \cite{bui2021growth}. The readers are invited to check \cite{bui2021growth} for different perspectives and examples of the definitions.

We are given a positive vector $s\in\mathbb R^d$ and a map $*:\mathbb R^d\times\mathbb R^d\to\mathbb R^d$ defined by nonnegative coefficients $c^{(k)}_{i,j}$ so that
\[
    (x*y)_k = \sum_{i,j} c^{(k)}_{i,j} x_i y_j
\]
for any $k$ and any vectors $x,y$. Note that throughout the article we always denote by $x_i$ the $i$-th entry of a vector $x$.

We denote by $A_n$ the set of all the results obtained by applying $n-1$ applications of $*$ to $n$ instances of $s$, that is: $A_1=\{s\}$ and for $n\ge 2$,
\[
	A_n=\bigcup_{1\le m\le n-1} \{x*y: x\in A_m, y\in A_{n-m}\}.
\]

The largest entry $g(n)$ over all the resulting vectors can be expressed as
\[
	g(n)=\max \{v_i: v\in A_n, 1\le i\le d\}.
\]

For convenience, we also denote by $g_k(n)$ the largest $k$-th entry over all the resulting vectors, that is
\[
	g_k(n)=\max \{v_k: v\in A_n\}.
\]
An obvious relation between $g(n)$ and $g_k(n)$ is $g(n)=\max_k g_k(n)$.

A pair $(*,s)$ is called a \emph{system} and the following limit $\lambda$ is called the \emph{growth rate of the system}
\[
	\lambda=\lim_{n\to\infty} \sqrt[n]{g(n)}.
\]

The limit was actually first proved in \cite{bui2021growth} by the notion of ``linear pattern'' (and will get proved again in Section \ref{sec:upper-bound} with another proof). Note that the limit $\lambda$ actually depends on the map and the vector of the considered system, but we do not denote it explicitly by $\lambda_{*,s}$ as they are known from context.

Before introducing linear patterns, we show a correspondence between the rooted binary trees of $n$ leaves and the combinations of $n$ instances of $s$. In one direction, we let the expression associated with a tree of a single leaf be the vector $s$ itself, and the expression associated with a tree of higher number of leaves be $L*R$, where $L,R$ are respectively the expressions associated with the left and right branches. The other direction is obvious as the previous association is a one-to-one mapping. Note that the map from the binary trees to the \emph{values} of the combinations is however not injective. Given any vector $v$ in $A_n$, we just pick any binary tree that gives $v$ to be the \emph{associated tree} with $v$. The arguments in this paper are independent of the choice. The other direction is deterministic: The vector that the binary tree gives is said to be the \emph{associated vector} with the tree. Note that from now on, all considered trees are rooted binary trees. In some places, we say the tree associated with $g_k(n)$ instead of saying the tree of $n$ leaves associated with a vector whose $k$-th entry is $g_k(n)$ for short. 

A \emph{linear pattern} $P=(T,\ell)$ is a pair of a tree $T$ and a marked leaf $\ell$ in $T$. Suppose in the expression associated with $T$, we put a vector variable $u$ instead of the fixed vector $s$ in the place associated with the leaf $\ell$. The value of the expression is then a vector variable $v$ linearly depending on $u$. Let $M=M(P)$ be the matrix representing the dependency, that is $v=Mu$. The matrix $M$ is said to be the \emph{associated matrix} with the linear pattern. A quick observation is that $M_{i,j}\le\const g_i(n)$ for any $i,j$ where $n$ is the number of leaves in $T$.

We would give an example of linear pattern from \cite{bui2021growth}: Let $s=(1,1)$ and $x*y=(x_1y_2 + x_2y_1, x_1y_2)$, then we have $g_1(n)=F_{n+1}$ and $g_2(n)=F_n$. The combinations that yield the Fibonacci numbers can be obtained from the linear pattern $P=(T,\ell)$ where $T$ has precisely two leafs with $\ell$ on the left. The golden ratio is actually the spectral radius of the matrix
\begin{equation*}
    M(P)=
    \begin{bmatrix}
        1 & 1\\
        1 & 0
    \end{bmatrix}
    .
\end{equation*}

The \emph{dependency graph} of a system, which expresses the dependency among the dimensions in $*$, is defined to be a directed graph that takes the dimensions as the vertices; there is an edge from $k$ to $i$ if there exists some $j$ so that either $c^{(k)}_{i,j}\ne 0$ or $c^{(k)}_{j,i}\ne 0$. As a directed graph, the dependency graph can be partitioned into strongly connected components, for which we call \emph{components} for short. We define a partial order between the components: A component $C_1$ is said to be less than a component $C_2$ for $C_1\ne C_2$ if there are vertices $i\in C_2$, $j\in C_1$ and a path from $i$ to $j$. For any $C_1,C_2$, we say $C_1\le C_2$ if $C_1<C_2$ or $C_1=C_2$.

\begin{proposition}
\label{prop:in-greatest-component}
	Suppose a component $C$ is greater than every other component, then $g_k(n)\ge\const g(n)$ for every $k\in C$.
\end{proposition}
\begin{proof}
	This is actually \cite[Lemma 5]{bui2021growth}.
\end{proof}
It means that $g_k(n)$ and $g(n)$ are in a constant factor of each other.

\begin{proposition}\label{prop:g_k-not-very-different}
	Given any vertex $k$, for any fixed $d>0$, the value $g_k(n)$ is at least a constant times $g_k(n+d)$.
\end{proposition}
\begin{proof}
	This is actually \cite[Corollary 1]{bui2021growth}.
\end{proof}
Note that however, the other direction of the inequality may not hold in some odd cases, for example, when $g_k(1)=s_k>0$ and $g_k(n)=0$ for $n\ge 2$ (e.g. $c^{(k)}_{i,j}=0$ for every $i,j$). There are some less trivial examples, but one may observe that it holds in general. A particular case is shown in Observation \ref{obs:bounded-distance}.

\begin{definition}\label{def:C-subsystem}
	Given a component $C$, the $C$-subsystem is the system with $(*',s')$ deduced from $(*,s)$ by restricting $(*,s)$ to only the dimensions in the components at most $C$. In particular, $s'_i=s_i$ for $i\in C'$ and $C'\le C$ (other dimensions $i$ are removed). Likewise, the coefficients $c'^{(k)}_{i,j}=c^{(k)}_{i,j}$ for $k\in C'$ and $C'\le C$.
\end{definition}

\begin{remark}\label{rem:same-component}
	With the viewpoint of the $C$-subsystem and Proposition \ref{prop:in-greatest-component}, we can conclude that the limit 
    \[
        \lambda_k=\lim_{n\to\infty} \sqrt[n]{g_k(n)}
    \]
    exists for every $k$ due to the existence of $\lambda=\lim_{n\to\infty} \sqrt[n]{g(n)}$. For two vertices $i,j$ in the same component $C$, the two values $g_i(n)$ and $g_j(n)$ are in a constant factor of each other, by applying Proposition \ref{prop:in-greatest-component} to the $C$-subsystem. It follows that $\lambda_i=\lambda_j$.
\end{remark}

We introduce some notations for convenience later.
\begin{definition}
	Given two patterns $P_1=(T_1,\ell_1)$ and $P_2=(T_2,\ell_2)$. The composition of the two patterns, denoted by $P_1\oplus P_2$, is the pattern $P=(T,\ell)$ where $T$ is obtained from $T_1$ by replacing $\ell_1$ by $T_2$, and $\ell$ is the leaf $\ell_2$ in this instance of $T_2$. A quick observation is $M(P_1\oplus P_2)=M(P_1) M(P_2)$.

	For a pattern $P=(T,\ell)$, we denote by $|P|$ the number of leaves excluding $\ell$ in $T$. We have $|P_1\oplus P_2|=|P_1|+|P_2|$ for any two patterns $P_1,P_2$.

	When we regard ``the number of leaves of pattern $P$'', we mean $|P|$. In most of the cases in this paper, the distinction between $|P|$ and the number of leaves in $T$ does not matter very much. (However, one must be careful when taking the root, it must be the $|P|$-th root.)

	We also denote by $P\oplus T'$ the tree obtained from the tree of the pattern $P$ by replacing the marked leaf by the tree $T'$.
\end{definition}
Since the composition $\oplus$ is defined, we would make a remark on the following decomposition, which is implicitly used throughout the article.
\begin{remark}
    Every pattern $P=(T,\ell)$ can be decomposed into $P=P_1\oplus\dots\oplus P_t$ so that the marked leaf $\ell_k$ of each pattern $P_k=(T_k,\ell_k)$ is a branch of the root of $T_k$. Indeed, let the path from the root of $T$ to $\ell$ be $v_0,\dots,v_t$. Each tree $T_k$ has the root $v_{k-1}$ with the marked leaf $\ell_k=v_k$ being one branch, and the other branch is what originally in $T$.
\end{remark}

The following observation is useful in various proofs.
\begin{observation} \label{obs:pattern-for-distance}
	If there is a path from $i$ to $j$ of length $d$, then there exists a pattern of $d$ leaves with the associated matrix $M$ satisfying $M_{i,j}>0$. It follows that $g_i(n+d)\ge\const g_j(n)$ for every $n$.
    On the other hand, if $M$ is the matrix associated with some pattern and $M_{i,j}>0$, then there is a path from $i$ to $j$.
\end{observation}
\begin{proof}
	Suppose there is an edge from $k$ to $i$ with $c^{(k)}_{i,j}>0$. Consider the pattern $(T,\ell)$ for a tree $T$ of two leaves with $\ell$ on the left. The associated matrix $M$ has $M_{k,i}>0$. A similar construction is for $c^{(k)}_{j,i}>0$ with the marked leaf $\ell$ on the right. Let the path from $i$ to $j$ be $k_0,k_1,\dots,k_d$. The desired pattern is $P_1\oplus\dots\oplus P_d$ where $P_t$ is the pattern constructed from the edge $k_{t-1}k_t$.

	Given such a pattern $P=P_1\oplus\dots\oplus P_d$, we can see that $g_i(n+d)\ge\const g_j(n)$ by considering the tree $P\oplus T^*$ where the tree $T^*$ is associated with $g_j(n)$.

	In the other direction, let $P$ be the pattern that $M$ with $M_{i,j}>0$ is associated with. Consider the decomposition $P=P_1\oplus\dots\oplus P_t$ so that each $P_k$ has the marked leaf being one branch of the root. Let $M_k$ be the associated matrix with $P_k$, we have $M=M_1\dots M_t$, that is
    \[
        M_{i,j}=\sum_{k_1,\dots,k_{t-1}} (M_1)_{i,k_1} (M_2)_{k_1,k_2} \dots (M_{t-1})_{k_{t-2},k_{t-1}} (M_t)_{k_{t-1},j}.
    \]    
    As $M_{i,j}>0$, there exist $k_1,\dots,k_{t-1}$ so that all 
    \[
        (M_1)_{i,k_1},(M_2)_{k_1,k_2},\dots,(M_{t-1})_{k_{t-2},k_{t-1}},(M_t)_{k_{t-1},j}
    \]
    are positive. It follows that there are edges $ik_1$, $k_1k_2$, $\dots$, $k_{t-2}k_{t-1}$, $k_{t-1}j$, which form the path $i,k_1,k_2,\dots,k_{t-1},j$.
\end{proof}

\section{A polynomial upper bound of $g(n)/\lambda^n$ and another proof of the limit $\lambda$}
\label{sec:upper-bound}
In order to prove the upper bound in Theorem \ref{thm:main} and the limit $\lambda$ at the same time, we use the following quantity
\[
	\theta=\sup_P \max_i \sqrt[|P|]{M(P)_{i,i}}.
\]

At first, we express $\theta$ as a lower bound for the growth of $g(n)$.
\begin{proposition}\label{prop:lower-bound-of-liminf}
	The following inequality holds
	\[
		\liminf_{n\to\infty} \sqrt[n]{g(n)}\ge\theta.
	\]
\end{proposition}
\begin{proof}
    It suffices to prove that for any $P$ and any $i$, we have $\liminf_{n\to\infty} \sqrt[n]{g(n)}\ge\sqrt[|P|]{M(P)_{i,i}}$. We assume $M(P)_{i,i}>0$, otherwise it is trivial. Considering the sequence $P_1,P_2,\dots$ where $P_1=P$ and $P_k=P_{k-1}\oplus P$ for $k\ge 2$, we have $M(P_t)=M(P)^t$. 
    For each $n=t|P|+r$ ($1\le r\le |P|$), let $v$ be the vector associated with the tree $P_t\oplus T_0$ where $T_0$ is the tree associated with $g_i(r)$. The conclusion follows from
    \begin{align*}
            g(n)&\ge v_i\ge M(P_t)_{i,i} g_i(r) \ge\const M(P_t)_{i,i} = \const (M(P)^t)_{i,i}\\
            &\ge \const (M(P)_{i,i})^t = \const \left(\sqrt[|P|]{M(P)_{i,i}}\right)^{t|P|}\ge \const \left(\sqrt[|P|]{M(P)_{i,i}}\right)^n. \qedhere
    \end{align*} 
\end{proof}

The following observation relates $\theta$ to the entries of the associated matrices.
\begin{observation}\label{obs:upper-bound}
	For any linear pattern $P=(T,\ell)$ with the associated matrix $M$ and $|P|=m$. If $i,j$ are in the same component, then
	\[
		M_{i,j}\le\const\theta^m.
	\]
\end{observation}
\begin{proof}
	We assume the component contains at least one edge, as it is trivial otherwise.
    By Observation \ref{obs:pattern-for-distance}, we have a linear pattern $P_0=(T_0,\ell_0)$ of a bounded number of leaves so that the associated matrix $M_0$ has $(M_0)_{j,i} > 0$. Let $P'=P\oplus P_0$ and $M'$ be the matrix associated with $P'$. We have ${M'}_{i,i}\ge M_{i,j} (M_0)_{j,i}\ge\const M_{i,j}$. On the other hand, ${M'}_{i,i}\le\theta^{m+O(1)}$, which is followed by $M_{i,j}\le\const\theta^m$.
\end{proof}

Now, we will investigate a more complicated case.
\begin{proposition}\label{prop:upper-bound-matrix-entry}
	Given a vertex $i$ with the condition that for every $j$ in a component lower than the component of $i$, there exists some $\alpha$ so that $g_j(m)\le\const m^\alpha \theta^m$ for every $m$. Suppose $k$ is a vertex in a component lower than the component of $i$ and $g_k(m)$ is not bounded by any fixed constant times $\theta^m$ for every $m$, then there exists some $r$ so that for every linear pattern $P$ with $|P|=n$ and the associated matrix $M$, we have
	\[
		M_{i,k}\le\const n^r\theta^n.
	\]
\end{proposition}
\begin{proof}
	Decompose $P$ into $P=P_1\oplus\dots\oplus P_t$ so that each pattern has the marked leaf being a branch of the root. Let $M_1,\dots,M_t$ be the matrices associated with $P_1,\dots,P_t$ respectively, we have $M=M_1\dots M_t$. Let $C$ be the component of $i$, the entry $M_{i,k}$ can be written as
	\begin{equation}\label{eq:decomposition}
		M_{i,k}=\sum_{\substack{1\le s\le t \\ j_1\in C,\,j_2\notin C}} (M_1\dots M_{s-1})_{i,j_1} (M_s)_{j_1,j_2} (M_{s+1}\dots M_t)_{j_2,k},
    \end{equation}
    where $j_1=i$ if $s=1$ and $j_2=k$ if $s=t$.

	We proceed by considering all the nonzero summands. This means $j_2$ is in a component directly lower than $C$, in order for $(M_s)_{j_1,j_2}$ to be nonzero. We can conclude right away that $(M_{s+1}\dots M_t)_{j_2,k}\le\const g_{j_2}(m)\le \const m^\alpha \theta^m$ where $m=|P_{s+1}|+\dots+|P_t|+1$ by the condition of the proposition. Also, by Observation \ref{obs:upper-bound}, we have $(M_1\dots M_{s-1})_{i,j_1}\le\const\theta^{|P_1|+\dots |P_{s-1}|}$ as $i,j_1$ are both in $C$.

	It remains to consider $(M_s)_{j_1,j_2}$. Suppose the marked leaf of $P_s$ is the left branch, without loss of generality. Let $v$ be the vector associated with the right branch of the tree of $P_s$. We have
    \[
        (M_s)_{j_1,j_2}=\sum_j c^{(j_1)}_{j_2, j} v_j.
    \]
    Note that in order for the summand in \eqref{eq:decomposition} to be nonzero, there is a path from $j_2$ to $k$, hence $g_{j_2}(m)$ is also not bounded by any fixed constant times $\theta^m$ for every $m$ by Observation \ref{obs:pattern-for-distance}. We can see that for any $j$ so that $c^{(j_1)}_{j_2, j} > 0$, the vertex $j$ is not in $C$. Indeed, assume otherwise, consider a linear pattern $\hat P$ where the left branch of $m$ leaves has the associated vector $u$ with $u_{j_2}=g_{j_2}(m) > K \theta^m$ for $K$ large enough and the right branch is the marked leaf. The associated matrix $\hat M$ has 
    \[
        \hat M_{j_1,j}=\sum_{j'} c^{(j_1)}_{j',j} u_{j'}\ge c^{(j_1)}_{j_2,j} u_{j_2}>\const K\theta^m
    \]
    contradicting Observation \ref{obs:upper-bound}. Therefore, for such $j$ we have $v_j\le g_j(|P_s|)\le \const |P_s|^{\alpha} \theta^{|P_s|}$ by the condition of the proposition. In other words,
    \[
        (M_s)_{j_1,j_2} = \sum_j c^{(j_1)}_{j_2, j} v_j \le \const |P_s|^{\alpha} \theta^{|P_s|}.
    \]

	In total, each summand in \eqref{eq:decomposition} is at most $\const n^{2\alpha} \theta^n$. Meanwhile, there are only at most linearly many options for $s$ and constantly many options for $j_1,j_2$. Therefore, for $r=2\alpha+1$, we have
	\[
		M_{i,k}\le\const n^r\theta^n.\qedhere
	\]
\end{proof}

The condition in Proposition \ref{prop:upper-bound-matrix-entry} is in fact not necessary by Proposition \ref{prop:upper-bound-by-theta} below.
The former proposition itself is actually a fact used in the induction step of the proof of the latter proposition.
\begin{proposition}\label{prop:upper-bound-by-theta}
	There exists $r$ so that
	\[
		g(n)\le\const n^r\theta^n.
	\]
\end{proposition}
\begin{proof}
	We prove an equivalent conclusion that: For every $i$ there exists $r$ so that 
	\[
		g_i(n)\le\const n^r\theta^n.
	\]

	Observation \ref{obs:upper-bound} is equivalent to the conclusion for any $i$ a minimal component $C$. Indeed, let $P=(T,\ell)$ where $T$ is the tree associated with $g_i(n)$ and $\ell$ is any leaf. Let $M$ be the associated matrix, we have $g_i(n)\le\const M_{i,j}$ for some $j$. Since $j$ is also in $C$, we have $g_i(n)\le\const M_{i,j}\le \const \theta^n$.

	Consider a component $C$, suppose the conclusion holds for any index $j$ in a component lower than $C$, we prove that it holds for any $i$ in $C$. 

	Consider a tree $T$ of $n$ leaves whose associated vector $v$ has $v_i=g_i(n)$. Pick any subtree $T_0$ of $m$ leaves so that $n/3\le m\le 2n/3$. Consider the decomposition $T=P'\oplus T_0$.
	Let $M'$ be the associated matrix of $P'$. 

	Let $u$ be the vector associated with $T_0$, we have $v=M'u$, that is $v_i=\sum_j M'_{i,j} u_j$. It follows that for some $k$, we have 
	\[
		g_i(n)\le\const M'_{i,k} u_k.
	\]

	We consider the following cases:
	\begin{itemize}
		\item
			If $k$ is in a lower component than $C$ but $g_k(t)$ is not bounded by any fixed constant times $\theta^t$ for every $t$, then by Proposition \ref{prop:upper-bound-matrix-entry}, we have ${M'}_{i,k}\le\const (n-m)^\alpha \theta^{n-m}$ for some $\alpha$.
			Suppose $g_k(t)\le\const t^\beta \theta^t$ for some $\beta$, we have $u_k\le g_k(m)\le \const m^\beta \theta^m$. In total,
			\[
				g_i(n)\le \const (n-m)^\alpha \theta^{n-m} m^\beta \theta^m\le \const n^{\alpha+\beta} \theta^n = \const n^\gamma\theta^n,
			\]
			where $\gamma=\alpha+\beta$.

		\item
			If $k$ is in a lower component than $C$ and $g_k(t)\le\const\theta^t$ for every $t$, then by $u_k\le g_k(m)\le \const \theta^m$ and ${M'}_{i,k}\le\const g_i(n-m+1)$, we have 
			\[
				g_i(n)\le \const g_i(n-m+1) \theta^m.
			\]

		\item
			If $k$ is in the component $C$, we have ${M'}_{i,k}\le\const\theta^{n-m}$ by Observation \ref{obs:upper-bound}. As $u_k\le g_k(m)$, we have 
			\[
				g_i(n)\le\const\theta^{n-m} g_k(m).
			\]
	\end{itemize}

	In any case of the two latter cases, we have reduced the size $n$ considerably by at least a fraction of $n$ to $n-m+1$ as in $g_i(n-m+1)$, or to $m$ as in $g_k(m)$. Note that $g_i(n-m+1)$ and $g_k(m)$ have $i,k$ both still in the component $C$. After repeating the process $O(\log n)$ times, and stopping only when the current $n$ is small enough or we fall into the first case, we obtain
	\[
		g_i(n)\le \const K^{O(\log n)} n^\gamma \theta^{n+O(\log n)},
	\]
	for some constant $K$.

	As $x^{\log n}=n^{\log x}$, the induction step finishes since for some $r$,
	\[
		g_i(n)\le\const n^r\theta^n
	\]

	The conclusion follows by induction.
\end{proof}

Now the existence of $\lambda$ is clear due to Proposition \ref{prop:lower-bound-of-liminf} and
\[
	\limsup_{n\to\infty} \sqrt[n]{g(n)}\le\limsup_{n\to\infty} \sqrt[n]{\const n^r\theta^n}=\theta.
\]

A corollary is $\lambda=\theta$.
\begin{corollary}\label{cor:new-formula}
	The limit $\lambda$ exists and can be expressed as
	\[
		\lambda=\sup_P \max_i \sqrt[|P|]{M(P)_{i,i}}.
	\]
\end{corollary}

Note that in \cite{bui2021growth}, it was shown that
\[
    \lambda=\sup_P \sqrt[|P|]{\rho(M(P))},
\]
where $\rho(M(P))$ is the spectral radius of $M(P)$. The readers can relate this to the fact: Given a nonnegative matrix $M$, we have
\[
    \rho(M) = \sup_n \max_i \sqrt[n]{(M^n)_{i,i}}
\]
(see \cite{bui2023bound} for the treatment of the more general notion of joint spectral radius). Actually, these two results can deduce the new formula in Corollary \ref{cor:new-formula}.

Another corollary is the upper bound for $g(n)$ in Theorem \ref{thm:main}.
\begin{theorem}\label{thm:upper-bound}
	There exists $r$ so that
	\[
		g(n)\le\const n^r\lambda^n.
	\]
\end{theorem}

\section{A polynomial upper bound for $\lambda^n/g(n)$}
\label{sec:lower-bound}
In this section, we prove the lower bound in Theorem \ref{thm:main}, as restated in the theorem below.

\begin{theorem}\label{thm:lower-bound}
	There exists $r$ so that
	\[
		g(n)\ge\const n^{-r}\lambda^n.
	\]
\end{theorem}

We first give a classification of components so that proving the lower bound of $g_k(n)$ for each $k$ in the component could be more convenient.
\begin{definition}
    A component $C$ is said to be \emph{strongly self-dependent} if there are three indices $k,i,j$ in $C$ (not necessarily different) so that $c^{(k)}_{i,j}>0$.
\end{definition}
\begin{definition}
	A component $C$ is said to be \emph{weakly self-dependent} if for every $k\in C$ and for any $c^{(k)}_{i,j}>0$ at least one of $i,j$ is in a lower component than $C$, and for any $j$ in a component lower than the component of $k$ we have $\lambda_j<\lambda_k$.
\end{definition}

Before classifying the remaining components, we give the following observation.
\begin{observation} \label{obs:depends-on-lower-components}
	Let $k$ be so that for each $c^{(k)}_{i,j}>0$, both $i,j$ are in a lower component than the component of $k$, we have 
	\[
		\lambda_k=\max_{i,j:\ c^{(k)}_{i,j}>0} \max\{\lambda_i,\lambda_j\}.
	\]
\end{observation}

Note that when there is no edge from $k$ (that is $c^{(k)}_{i,j}=0$ for every $i,j$), we have $\lambda_k=0$ with the convention that the maximum of an empty list is zero.

\begin{corollary}\label{cor:one-same-one-lower}
    For any $k$ in a weakly self-dependent component $C$, there exist $i,j$ so that $c^{(k)}_{i,j}>0$ and one of $i,j$ is in $C$ while the other one is in a component lower than $C$.
\end{corollary}

The remaining components $C$ (other than strongly self-dependent and weakly self-dependent components) satisfy (i) for every $k\in C$ and for any $c^{(k)}_{i,j}>0$ at least one of $i,j$ is in a lower component than $C$ and (ii) there exists $k\in C$ and $j$ in a lower component than $C$ so that $\lambda_k=\lambda_j$. It means the remaining components are included in the following class.
\begin{definition}
    A component $C$ is said to be \emph{not self-dependent} if for each $k\in C$ we have $\lambda_k=\lambda_j$ for some $j$ in a lower component than $C$.
\end{definition}

In total, we have the following proposition.
\begin{proposition}\label{prop:classes-of-components}
   The three classes of components: strongly self-dependent, weakly self-dependent and non-self-dependent components cover all the components.
\end{proposition}

The classification allows a convenient study of $g_k(n)$ for each type of components. In particular, we prove that $g_k(n)\ge\const n^{-r}{\lambda_k}^n$ for some $r$ if $k$ is in a strongly self-dependent component (in Theorem \ref{thm:lower-bound-connected-dependency}), and $g_k(n)\ge\const {\lambda_k}^n$ if $k$ is in a weakly self-dependent component (in Theorem \ref{thm:submulti-for-half-self-dependency}).

\subsection{Strongly self-dependent components}
We show that $g_k(n)$ is weakly supermultiplicative for $k$ in a strongly self-dependent component.
\begin{theorem}\label{thm:upper-bound-almost-connected}
    Let $k$ be in a strongly self-dependent component with $c^{(k)}_{i,j}>0$. Then for any $m,n$ we have
	\[
		g_k(m+n)\ge\const g_k(m) g_k(n).
	\]
\end{theorem}
\begin{proof}
	Using the results in Section \ref{sec:preliminary}, we have
	\[
		g_k(m) g_k(n)\le\const g_i(m+d_1) g_j(n+d_2)\le \const g_k(m+n+d_1+d_2)\le\const g_k(m+n),
	\]
	where $d_1,d_2$ are the distances from $i$ to $k$ and from $j$ to $k$, respectively. The first inequality is due to Observation \ref{obs:pattern-for-distance}, while the last inequality is due to Proposition \ref{prop:g_k-not-very-different}. The middle inequality is obtained by considering a tree where the left branch is associated to $g_i(m+d_1)$ and the right branch is associated to $g_j(n+d_2)$.

	An alternate argument is
	\[
		g_k(m) g_k(n)\le\const g_i(m) g_j(n)\le \const g_k(m+n),
	\]
	where the first inequality is obtained by applying Proposition \ref{prop:in-greatest-component} as in Remark \ref{rem:same-component}.
\end{proof}

Corollaries of the result include the limit of $\sqrt[n]{g_k(n)}$ and $g_k(n)\le\const {\lambda_k}^n$ by applying Fekete's lemma \cite{fekete1923verteilung} to the supermultiplicative sequence $\{\const g_k(n)\}_n$. The upper bound is a case of Theorem \ref{thm:upper-bound} as we reduce the polynomial $n^r$ to $n^0=1$. Note that when the dependency graph is connected, the only component is strongly self-dependent, for which the above approach gives a simple proof of the limit. We give some further discussion about the proof of Theorem \ref{thm:upper-bound-almost-connected} as below, which can be safely skipped.

\begin{remark}
The two arguments in the proof both use not so trivial propositions. We can avoid using them and obtain a bit weaker result, which can however still show the bound of $g_k(n)$ and the limit of $\sqrt[n]{g_k(n)}$.
Indeed, after obtaining the inequality $g_k(m)g_k(n)\le K g_k(m+n+d_1+d_2)$ for some constant $K$ as in the first half of the first argument, we shift the sequence and multiply both sides by $K$ to get
\[
	K g_k(m-d_1-d_2) K g_k(n-d_1-d_2)\le K g_k(m+n-d_1-d_2).
\]
The sequence $s_n=K g_k(n-d_1-d_2)$ is supermultiplicative. Although the sequence $s_n$ is undefined for some beginning elements, (a variant of) Fekete's lemma still applies and we have $\sqrt[n]{s_n}$ converges to $\lambda_k=\sup_n \sqrt[n]{s_n}$. The original sequence $\sqrt[n]{g_k(n)}$ also converges to $\lambda_k$. The bound also follows.

In fact, we can deduce the upper bound $g_k(n)\le\const {\lambda_k}^n$ from Corollary \ref{cor:new-formula} as follows. Suppose otherwise, that is we have $g_k(n)=K{\lambda_k}^n$ for a large enough $K$. It follows that we have $g_j(n+d)\ge\const K{\lambda_k}^n$ where $d$ is the distance from $j$ to $k$ by Observation \ref{obs:pattern-for-distance}. Let the tree associated with $g_j(n+d)$ be $T_r$. We consider the pattern $P'=(T',\ell')$ where $T'$ is the tree taking $T_r$ as the right branch and the single $\ell'$ as the left branch. Let $M'$ be the matrix associated with $P'$, we have $M'_{k,i}\ge c^{(k)}_{i,j} g_j(n+d)\ge\const K{\lambda_k}^n$. Let $P_0=(T_0,\ell_0)$ be the pattern of a bounded number of leaves so that the associated matrix $M_0$ has $(M_0)_{i,k}>0$ (also by Observation \ref{obs:pattern-for-distance}). The matrix $M$ associated with $P=P'\oplus P_0$ has $M_{k,k}\ge M'_{k,i} (M_0)_{i,k}\ge\const K{\lambda_k}^n$. As $|P|-n$ is bounded, we have $\sqrt[|P|]{M_{k,k}} > \lambda_k$ when $K$ is large enough. This contradicts to Corollary \ref{cor:new-formula} when we consider the $C$-subsystem for the component $C$ of $k$, where $\lambda_k$ is the growth rate $\lambda$ for the new system. 
\end{remark}

The upper bound $\const{\lambda_k}^n$ is a nice corollary of Fekete's lemma. We naturally wonder what a lower bound would be, whether the leading constant should be replaced by something arbitrarily small. Actually, we would conjecture that $g_k(n)\ge\const{\lambda_k}^n$. However, what we could come up in the general case is just the following result. It is also a corollary of Fekete's lemma, but for a submultiplicative sequence. The interesting point is that the supermultiplicative form as in Theorem \ref{thm:upper-bound-almost-connected} is used in the proof.

\begin{theorem}\label{thm:lower-bound-connected-dependency}
	If $k$ is in a strongly self-dependent component, then there exists some $r$ so that
	\[
		g_k(n)\ge\const n^{-r} {\lambda_k}^n.
	\]
\end{theorem}

We first give a lemma, which is an extension of the technique used in the proof of \cite[Lemma 1]{bui2021growth}.
\begin{lemma} \label{lem:extension}
	Let $T$ be a tree of $n$ leaves associated to $g(n)$. If there is a subtree $T_0$ of $m$ leaves, then
	\[
		g(n)\le\const g(m) g(n-m+1).
	\]
\end{lemma}
\begin{proof}
	Let $T'$ be the tree obtained from $T$ after contracting the whole $T_0$ into a single leaf $\ell$. The tree $T'$ has $n-m+1$ leaves.
	Let $v,u,v'$ be the associated vectors with $T,T_0,T'$, respectively. Let $M$ be the matrix associated with the pattern $(T',\ell)$. Clearly, $v=Mu$ and $v'=Ms$. Suppose $g(n)=v_j$, we have
	\[
		g(n)=v_j=(Mu)_j \le \left(\max_i \frac{u_i}{s_i}\right) (Ms)_j \le \const g(m) v'_j \le \const g(m)g(n-m+1).\qedhere
	\]
\end{proof}

\begin{corollary} \label{cor:extension-for-k}
	Given some $k$, let $T$ be a tree of $n$ leaves associated to $g_k(n)$. If there is a subtree $T_0$ of $m$ leaves, then
	\[
		g_k(n)\le\const g_k(m) g_k(n-m).
	\]
\end{corollary}
\begin{proof}	
	At first, for every $n$, we have $g(n+1)\le\const g(n)$. This inequality can be deduced from \cite[Lemma 1]{bui2021growth}, which states that for every $k$, we have $g_k(n)\ge\const g_k(n+1)$. It is also a special case of Lemma \ref{lem:extension} itself when we take any subtree $T_0$ of $2$ leaves, which always exists. 

	It follows that with the same condition as in Lemma \ref{lem:extension}, we also get 
	\[
		g(n)\le\const g(m) g(n-m).
	\]
	The conclusion can now be obtained by applying Proposition \ref{prop:in-greatest-component} to the $C$-subsystem for the component $C$ of $k$.
\end{proof}
Note that this corollary actually deduces Proposition \ref{prop:g_k-not-very-different} by setting $m=1$ and applying it several times.

The following observation is useful later.
\begin{observation} \label{obs:bounded-distance}
	If $k$ is in a component that has at least one edge inside (loops are also counted), then for every $d$, we have
	\[
		g_k(n+d)\ge\const g_k(n).
	\]
	It follows that $g_k(n)$ and $g_k(n+d)$ are in a constant factor of each other (by Proposition \ref{prop:g_k-not-very-different}).
\end{observation}
\begin{proof}
	By the condition of the component of $k$, for every $d$, there exists a pattern $P=(T,\ell)$ with $|P|=d$ so that the associated matrix $M$ has $M_{i,k}>0$ for some $i$ in the component of $k$ (by Observation \ref{obs:pattern-for-distance}). Let $T^*$ be the tree associated with $g_k(n)$. We can see that the $i$-th entry $v_i$ of the vector associated with $P\oplus T^*$ is at least a constant times $g_k(n)$. The conclusion follows since $g_k(n+d)\ge\const g_i(n+d)\ge\const v_i\ge\const g_k(n)$ with the first inequality obtained from Remark \ref{rem:same-component}.
\end{proof}
The condition in Observation \ref{obs:bounded-distance} can be relaxed, but the current form is enough for later applications.

\begin{proposition}\label{prop:weak-submulti}
	Suppose $k$ is in a strongly self-dependent component. For any $m,n$, we have
	\[
		g_k(m+n)\le \const K^{\log m} g_k(m) g_k(n),
	\]
	where $K$ is a constant.
\end{proposition}
\begin{proof}
	Before proving, we observe that given any $d_0\ge 1/2$, for any tree $T$ of at least $d_0$ leaves, there is always a subtree of $d$ leaves so that $d_0\le d\le 2d_0$ (left as an exercise for the readers).

	Let $T$ be the tree associated with $g_k(m+n)$. By the observation at the beginning, there is a subtree $T_0$ of $m_0$ leaves so that $m/2 \le m_0 \le m$. That means $g_k(m+n)\le \const g_k(m_0) g_k(n+m-m_0)$ by Corollary \ref{cor:extension-for-k}.
    We continue the process with a subtree of $m_1$ leaves so that $\frac{m-m_0}{2}\le m_1\le m-m_0$ from the tree associated with $g_k(n+m-m_0)$, for which $g_k(n+m-m_0)\le\const g_k(m_1)g_k(n+m-m_0-m_1)$.
	Repeating this process some $t=O(\log m)$ times, we obtain
	\[
		g_k(m+n)\le {K_0}^{\log m} g_k(m_0)\dots g_k(m_t) g_k(n),
	\]
	where $K_0$ is a constant, $m_0+\dots+m_t=m$. 

	Using the assumption of the component $C$, that is $g_k(a+b)\ge\const g_k(a)g_k(b)$ for any $a,b$ by Theorem \ref{thm:upper-bound-almost-connected}, we have
	\[
		g_k(m+n)\le K^{\log m} g_k(m_0+\dots+m_t) g_k(n) = \const K^{\log m} g_k(m) g_k(n),
	\]
	where $K$ is a constant.
\end{proof}
\begin{remark}
    In fact, a similar technique is also used in proving an asymptotic lower bound on the number of polyominoes in \cite{bui2023asymptotic}.
\end{remark}

We are now ready to prove Theorem \ref{thm:lower-bound-connected-dependency}.
\begin{proof}[Proof of Theorem \ref{thm:lower-bound-connected-dependency}]
	For any pair of $m,n$ with $m\le n$, Proposition \ref{prop:weak-submulti} gives
	\[
		g_k(m+n)\le \const K^{\log m} g_k(m) g_k(n),
	\]
	which means
	\[
		K^{\log (m+n)} g_k(m+n) \le \alpha K^{\log m} g_k(m) K^{\log n} g_k(n),
	\]
	for some constant $\alpha$, because $m+n$ and $n$ are in a constant factor of each other (note that $m\le n$).

	Writing differently, $K^{\log n}=n^{\log K}=n^r$ for $r=\log K$, and multiplying both sides of the inequality by $\alpha$, we have
	\[
		\alpha (m+n)^r g_k(m+n) \le \alpha m^r g_k(m) \alpha n^r g_k(n).
	\]

	Applying Fekete's lemma for the submultiplicative sequence $\alpha n^r g_k(n)$, we have
	\[
		\lambda_k=\inf_n \sqrt[n]{\alpha n^r g_k(n)},
	\]
	which means
	\[
		g_k(n)\ge\const n^{-r} {\lambda_k}^n.\qedhere
	\]
\end{proof}

\subsection{Weakly self-dependent components}
Let us start with the key step stating that the tree associated with $g_k(n)$ for $k$ in a weakly self-dependent component cannot be too arbitrary.
\begin{proposition}\label{prop:bounded-branch}
	Let $k$ be in a weakly self-dependent component. Let $T$ be a tree of $n$ leaves with the associated vector $w$. If the numbers of leaves in both branches of $T$ are unbounded, then $g_k(n)/w_k$ is unbounded.
\end{proposition}

Before proving the proposition, we give a lemma, which already appeared in an implicit (and slightly weaker) form in the proof of \cite[Theorem 1]{bui2021growth}. 
\begin{lemma} \label{lem:a-pattern-with-large-value}
	Consider a vertex $i$ such that every $j$ of a lower component than the component of $i$ has $\lambda_j<\lambda_i$. For each $\epsilon>0$, for any $d$ large enough, there exists a pattern of $m$ leaves so that $d\le m\le 2d+O(1)$ and the associated matrix $M$ has $M_{i,i}\ge\const (\lambda_i - \epsilon)^m$.
\end{lemma}

\begin{proof}
	Let us pick some $\epsilon^*>0$ and consider some $n$ large enough so that $(\lambda_k-\epsilon^*)^m < g_k(m) < (\lambda_k+\epsilon^*)^m$ for every $k$ and for every $m\ge n/3$. Let $T$ be the tree associated with $g_i(n)$. Take a subtree $T_0$ of $m$ leaves so that $n/3\le m\le 2n/3$. Consider the decomposition $T=P'\oplus T_0$. Let $M'$ be the matrix associated with $P'$.

	Let $u$ be the vector associated with $T_0$. Since $g_i(n)=(M'u)_i=\sum_j M'_{i,j} u_j$, there exists some $k$ so that 
	\[
		M'_{i,k} u_k\ge\const g_i(n).
	\]

	The vertex $k$ cannot be in a component lower than $i$ when $n$ is large enough and $\epsilon^*$ is small enough, since $u_k \le g_k(m)\le  (\lambda_k+\epsilon^*)^m$ is then too small, which makes
	\[
		M'_{i,k} \ge \const\frac{g_i(n)}{u_k} \ge \const\frac{(\lambda_i-\epsilon^*)^n}{(\lambda_k+\epsilon^*)^m}
	\]
	not bounded by any constant times $g_i(n-m+1)\le(\lambda_i+\epsilon^*)^{n-m+1}$. 

	If $k$ is in the same component as $i$, we have $\lambda_k=\lambda_i$, hence 
	\[
		M'_{i,k}\ge\const \frac{(\lambda_i-\epsilon^*)^n}{(\lambda_i+\epsilon)^m}.
	\]

	For any $\epsilon>0$, there exists an $\epsilon^*$ small enough and an $n$ large enough so that $M'_{i,k}\ge (\lambda_i-\epsilon)^{n-m}$. As $k$ is in the same component as $i$, we can extend the pattern $P'$ to a new pattern $P^*=P'\oplus \bar P$ where $\bar P$ is a pattern so that $M(\bar P)_{k,i}>0$ and $|\bar P|$ is the distance from $k$ to $i$ (by Observation \ref{obs:pattern-for-distance}). The number of leaves $m^*$ in the new pattern is at most bounded larger, therefore, the associated matrix $M^*$ has $(M^*)_{i,i}\ge\const M'_{i,k}\ge \const (\lambda_i-\epsilon)^{m^*}$.

	The conclusion follows, by choosing an $\epsilon^*$ small enough and $n=\lceil 3d\rceil$.
\end{proof}

\begin{remark}
	We can deduce a stronger result than Lemma \ref{lem:a-pattern-with-large-value} from Corollary \ref{cor:new-formula}: Since the sequence $\{\sup_{P:|P|=n} M(P)_{i,i}\}_n$ for each $i$ is supermultiplicative, it follows that the subsequence of all positive elements converges to a limit (by a variant of Fekete's lemma for nonnegative sequences in \cite{bui2023bound}). When $i$ satisfies the condition of Lemma \ref{lem:a-pattern-with-large-value}, the limit is $\lambda_i$. Details of deductions are left to the readers as an exercise.
\end{remark}

Now we prove Proposition \ref{prop:bounded-branch}.
\begin{proof}[Proof of Proposition \ref{prop:bounded-branch}]
    Suppose the contrary that the numbers of leaves of both branches are unbounded but $g_k(n)/w_k$ is bounded.
    
	Let $u,v$ be the vectors associated with the left and the right branches. We can assume that the number $m$ of leaves in the right branch is smaller. Choose some $\epsilon>0$ small enough, we suppose $m$ is large enough so that for every $m'\ge m/3 - O(1)$, we have $(\lambda_i-\epsilon)^{m'}<g_i(m')<(\lambda_i+\epsilon)^{m'}$ for every $i$. (One can later see in the proof what exactly $O(1)$ is.)

	Let $w$ be the vector associated with $T$.
	Since $w_k = \sum_{i,j} c^{(k)}_{i,j} u_i v_j$, we have
    \[
        w_k\le\const u_i v_j
    \]
    for some $i,j$.

	By Corollary \ref{cor:one-same-one-lower}, let $i^*, j^*$ be a pair so that $c^{(k)}_{i^*,j^*}>0$ with one of $i^*,j^*$ in $C$, say $i^*\in C$, where $C$ is the component of $k$. 
	Consider the tree $T^*$ where the left branch is the tree associated with $g_{i^*}(n-1)$ and the right branch is just a leaf.
	It follows that the associated vector $w^*$ to $T$ has $w^*_k\ge\const g_{i^*}(n-1)\ge\const (\lambda_k-\epsilon)^n$.

	Back to the tree $T$, we have
	\[
		u_iv_j\le\const g_i(n-m) g_j(m)\le\const (\lambda_i+\epsilon)^{n-m} (\lambda_j+\epsilon)^m.
	\]
	 It follows that $i\in C$ (hence $j\notin C$) when $\epsilon$ is small enough and $n$ is large enough, since otherwise, $u_iv_j$ is smaller than $(\lambda_k-\epsilon)^n\le \const w^*_k$ by an unbounded number of times.

	We now prove that $m$ being large enough raises a contradiction to the boundedness of $g_k(n)/w_k$, which finishes the proof.

	By Lemma \ref{lem:a-pattern-with-large-value}, there exists a pattern $P^*$ of $m^*$ leaves so that $m/3\le m^* \le 2m/3 + O(1)$ and $(M^*)_{i,i} \ge \const (\lambda_i-\epsilon)^{m^*}$ for the associated matrix $M^*$.

	Denoting $m'=m-m^*$, we proceed with transforming the original tree. As $m'\ge m/3-O(1)$, we have $g_j(m')\ge (\lambda_j-\epsilon)^{m'}$. Also, we have $g_j(m)\le(\lambda_j+\epsilon)^m$.
	At first, we replace the right branch by the tree associated with $g_j(m')$. After that, we replace the original left branch $\mathcal L$ by $P^*\oplus\mathcal L$. The new tree has the $k$-th entry of the associated vector at least
	\[
		c^{(k)}_{i,j} (M^*)_{i,i} u_i g_j(m') \ge \const (\lambda_i - \epsilon)^{m^*} u_i (\lambda_j - \epsilon)^{m'},
	\]
	which is greater than $w_k$ an unbounded number of times when $\epsilon$ is small enough and $m$ is large enough. That is because $w_k$ is actually bounded by a constant times
	\[
		u_i v_j\le u_i g_j(m)\le u_i (\lambda_j+\epsilon)^m,
	\]
	and the unboundedness of
	\[
		\frac{(\lambda_i - \epsilon)^{m^*} u_i (\lambda_j - \epsilon)^{m'}}{u_i (\lambda_j+\epsilon)^m} = \left(\frac{\lambda_i - \epsilon}{\lambda_j+\epsilon}\right)^{m^*} \left(\frac{\lambda_j-\epsilon}{\lambda_j+\epsilon}\right)^{m'}
	\]
	is due to $m^*$ and $m'$ in a constant factor of each other and the ratio between $(\lambda_i-\epsilon)/(\lambda_j+\epsilon)$ and $(\lambda_j+\epsilon)/(\lambda_j-\epsilon)$. 
\end{proof}

\begin{corollary} \label{cor:big-subtree-small-branch}
	Suppose $k$ is in a weakly self-dependent component. Let $T$ be the tree associated with $g_k(n)$. Then every subtree of $T$ with at least $n/2$ leaves has a branch with a bounded number of leaves.
\end{corollary}
Actually the same holds for every $cn$ with a fixed constant $0<c<1$. The constant $1/2$ is chosen for later usage.
\begin{proof}
	Let $T_0$ be a subtree of $m\ge n/2$ leaves. Consider the decomposition $T=P'\oplus T_0$. Let $M'$ be the matrix associated with $P'$. Let $v,u$ be the associated vectors with $T,T_0$, respectively. We have $v=M'u$. It follows that $v_k=(M'u)_k=\sum_j M'_{k,j} u_j$. It means for some $j$,
	\[
		v_k\le\const M'_{k,j} u_j.
	\]

	Applying Theorem \ref{thm:upper-bound} to the $C'$-subsystem with the support of Proposition \ref{prop:in-greatest-component} for the component $C'$ of $j$, we have $u_j\le g_j(m)\le \const m^{r_1}{\lambda_j}^m$ for some $r_1$.
	Also, we have $M'_{k,j}\le\const g_k(n-m+1)\le\const (n-m+1)^{r_2}{\lambda_k}^{n-m+1}$ for some $r_2$ (as the tree of $P'$ has $n-m+1$ leaves). In total,
    \[
        v_k\le\const m^{r_1} (n-m+1)^{r_2} {\lambda_j}^{m} {\lambda_k}^{n-m+1}.
    \]

	Suppose $j$ is not in the component of $k$, that is $\lambda_j < \lambda_k$. For some $\epsilon>0$ small enough, any $n$ large enough satisfies $g_k(n)\ge(\lambda_k-\epsilon)^n$. However, any $m$ large enough would make $v_k$ less than $(\lambda_k-\epsilon)^n$, contradiction.

	Therefore, both $j,k$ are in the same component. By Proposition \ref{prop:bounded-branch}, the entry $u_j$ can be increased by an unbounded number of times by a transform if both of the branches of $T_0$ are large enough. The conclusion follows, due to the maximality of $g_k(n)$.
\end{proof}

We now give a bound of $g_k(n)$ for $k$ in a weakly self-dependent component.
\begin{theorem} \label{thm:submulti-for-half-self-dependency}
	Given a vertex $k$ in a weakly self-dependent component, then the sequence $g_k(n)$ is weakly submultiplicative. It follows that $g_k(n)\ge\const{\lambda_k}^n$.
\end{theorem}
\begin{proof}
	Let $T$ be the tree associated to $g_k(n)$.
	By Corollary \ref{cor:big-subtree-small-branch}, there exist a leaf $\ell$ and a decomposition $P=P_1\oplus\dots\oplus P_{t-1}\oplus P_t$ for $P=(T,\ell)$ so that each $|P_i|$ for $i\ne t$ is bounded and the number of leaves in $P_t$ is at most $n/2$. The leaf $\ell$ is chosen by going from the root in the bigger branch at each step, and when we reach the first subtree of at most $n/2$ leaves, we let that subtree be the tree of $P_t$ and assign an arbitrary leaf in that subtree to $\ell$.

	If $P=P'\oplus P''$, by Corollary \ref{cor:extension-for-k} and Observation \ref{obs:bounded-distance}, we have 
	\[
		g_k(|P|)\le\const g_k(|P'|) g_k(|P''|).
	\]

	By the decomposition $P=P_1\oplus\dots\oplus P_{t-1}\oplus P_t$, for each $m\ge n/2$, there are $P', P''$ so that $P=P'\oplus P''$, and $|P'|-(n-m)$ and $|P''|-m$ are bounded. It follows from Observation \ref{obs:bounded-distance} that
	\[
		g_k(n)\le\const g_k(|P|)\le\const g_k(|P'|) g_k(|P''|)\le\const g_k(n-m) g_k(m).
	\]

	Let the final constant be $K$, we have $Kg_k(n)\le K g_k(n-m) Kg_k(m)$, i.e. the sequence $\{Kg_k(n)\}_n$ is submultiplicative. That is $\lambda_k=\inf_n \sqrt[n]{Kg_k(n)}$. The conclusion follows.
\end{proof}

\subsection{Conclusion for all components}
Now Theorem \ref{thm:lower-bound} is clear.
\begin{proof}[Proof of Theorem \ref{thm:lower-bound}]
	Consider a minimal component $C$ so that $\lambda_k=\lambda$ for $k\in C$. The minimality means that every component $C'$ lower than $C$ has $\lambda_{k'}<\lambda$ for $k'\in C'$. In other words, the component $C$ cannot be a non-self-dependent component. By the coverage of the components in Proposition \ref{prop:classes-of-components}, the component $C$ is either a strongly self-dependent component or a weakly self-dependent component. It follows from Theorem \ref{thm:lower-bound-connected-dependency} and Theorem \ref{thm:submulti-for-half-self-dependency} that $g_k(n)\ge\const n^{-r}{\lambda_k}^n$ for some $r$. The conclusion in Theorem \ref{thm:lower-bound} follows. 
\end{proof}

\bibliographystyle{unsrt}
\bibliography{gbm2}

\end{document}